\documentclass[12pt,fleqn,leqno]{amsart}

\usepackage{amssymb}
\usepackage{microtype}
\usepackage{thmtools}

\usepackage[svgnames]{xcolor}
\usepackage[colorlinks,linkcolor=blue,citecolor=Green]{hyperref}
\usepackage{titlesec}
\usepackage{fourier}
\usepackage[nameinlink]{cleveref}
\usepackage{enumitem}

\usepackage[a4paper,centering]{geometry}

\titleformat{\section}[block]
 {\bfseries}
 {\thesection.}
 {\fontdimen2\font}
 {}

 \setlist{noitemsep}
\setenumerate{labelindent=\parindent,label=\upshape{(\alph*)}}

\newtheorem{theorem}{Theorem}[section]
\newtheorem{corollary}[theorem]{Corollary}
\newtheorem{proposition}[theorem]{Proposition}

\theoremstyle{definition}
\newtheorem{remark}[theorem]{Remark}
\newtheorem{example}[theorem]{Example}
\newtheorem{definition}[theorem]{Definition}
\newtheorem{question}{Question}

\DeclareMathOperator{\N}{\mathbb{N}}
\DeclareMathOperator{\R}{\mathbb{R}}
\DeclareMathOperator{\uhr}{\upharpoonright}
\DeclareMathOperator{\cut}{ct}
\DeclareMathOperator{\noncut}{nct}
\DeclareMathOperator{\ord}{ord}
\DeclareMathOperator{\pend}{end}

\DeclareMathAlphabet{\mathpzc}{OT1}{pzc}{m}{it}
\DeclareMathOperator{\sel}{\mathpzc{V\mkern-4mu_{cs}}}
\DeclareMathOperator{\ocs}{\mathcal{O}\mathpzc{\mkern-4mu_{cs}}}

\numberwithin{equation}{section}
\begin{document}

\author{Valentin Gutev}

\address{Institute of Mathematics and Informatics, Bulgarian Academy
  of Sciences, Acad. G. Bonchev Street, Block 8, 1113 Sofia, Bulgaria}

\email{\href{mailto:gutev@math.bas.bg}{gutev@math.bas.bg}}



\title{On the Variety of Hyperspace Selections}

\begin{abstract}
  If $f$ is a continuous selection for the Vietoris hyperspace
  $\mathcal{F}(X)$ of the nonempty closed subsets of a space $X$, then
  the point $p=f(X)\in X$ is not as arbitrary as it might seem at
  first glance. In fact, the set $\ocs(X)$ of all these points reveals
  certain information about the variety of Vietoris continuous
  selections for $\mathcal{F}(X)$. Thus, for a connected space $X$, we
  will show that every point $p\in \ocs(X)$ is not only noncut, but
  also an endpoint of $X$. Another result of this paper is
  that in an arbitrary topological space $X$, the closure of the set
  $\ocs(X)$ is always a totally disconnected subset. Furthermore, we
  will also show that $\ocs(X)$ is a closed subset of every first
  countable totally disconnected space $X$. 
\end{abstract}

\date{\today}
\maketitle

\section{Introduction}

All spaces in this paper are infinite Hausdorff topological
spaces. Let $\mathcal{F}(X)$ be the set of all nonempty closed subsets
of a space $X$. We endow $\mathcal{F}(X)$ with the \emph{Vietoris
  topology} $\tau_V$, and call it the \emph{Vietoris hyperspace} of
$X$. Let us recall that $\tau_V$ is generated by all collections of
the form
\[
  \langle\mathcal{V}\rangle = \left\{S\in \mathcal{F}(X) : S\subseteq
    \bigcup \mathcal{V}\ \ \text{and}\ \ S\cap V\neq \emptyset\
    \text{for all}\ V\in \mathcal{V}\right\},
\]
where $\mathcal{V}$ runs over the finite families of open subsets of
$X$. A map $f:\mathcal{D}\to X$ is a \emph{selection} for a subset
$\mathcal{D}\subseteq \mathcal{F}(X)$ if $f(S)\in S$ for every
$S\in\mathcal{D}$. A selection $f:\mathcal{D}\to X$ is
\emph{continuous} if it is continuous with respect to the relative
Vietoris topology on $\mathcal{D}$, and we will use
$\sel[\mathcal{D}]$ to denote the set of all \emph{Vietoris continuous
  selections} for $\mathcal{D}$.\medskip

Most of the considerations in this paper will be related to the subset
$\ocs(X)\subseteq X$ of a space $X$ defined by
\begin{equation}
  \label{eq:Sequences-v2:1}
  \ocs(X)=
  \left\{f(X): f\in \sel[\mathcal{F}(X)]\right\}. 
\end{equation}
Of course, such considerations make sense when this set is nonempty,
i.e.\ when $\sel[\mathcal{F}(X)]\neq\emptyset$, in which case
$\ocs(X)$ can be regarded as the $X$-`Orbit' with respect to the
`action' of $\sel[\mathcal{F}(X)]$ on the hyperspace
$\mathcal{F}(X)$. Local properties of the elements of $\ocs(X)$ were
studied in \cite{Gutev2024a}, where this set was simply denoted by
$X_\Theta$. In this paper, we will study properties of $\ocs(X)$ as a
subspace of the space $X$.\medskip

The set $\ocs(X)$ can reveal certain information about the variety of
Vietoris continuous selections for $\mathcal{F}(X)$, and is not as
arbitrary as it might seem at first glance. In fact, $\ocs(X)$ was
implicitly present in several known results. To this end, let us
recall that a space $X$ is \emph{orderable} (or \emph{linearly
  ordered}\,) if it is endowed with the open interval topology
generated by some linear order on $X$, called \emph{compatible} for
$X$. Subspaces of orderable spaces are not necessarily orderable, they
are termed \emph{suborderable}. A space $X$ is \emph{weakly orderable}
if there exists a coarser orderable topology on $X$ with respect to
some linear order on it (called \emph{compatible} for $X$). The weakly
orderable spaces were introduced by Eilenberg \cite{eilenberg:41}
under the name of ``ordered'' topological spaces, and are often called
``Eilenberg orderable''. In the same paper, Eilenberg showed in
\cite[Theorem II]{eilenberg:41} that each connected weakly orderable
space has precisely two compatible linear orders which are inverse of
each other.\medskip

For a space $X$, let
$\mathcal{F}_2(X)=\left\{S\in \mathcal{F}(X): |S|\leq 2\right\}$. A
selection $\sigma:\mathcal{F}_2(X)\to X$ is commonly called a
\emph{weak selection} for $X$. In \cite[Definition 7.1]{michael:51},
Michael identified each weak selection $\sigma$ for $X$ with an
order-like relation $\leq_\sigma$ on $X$ defined by $x\leq_\sigma y$
whenever $\sigma(\{x,y\})=x$. The binary relation $\leq_\sigma$ on $X$
thus defined is both \emph{total} and \emph{antisymmetric}, but may
not be \emph{transitive} even when $\sigma$ is continuous. However,
when $X$ is connected and $\sigma$ is continuous, Michael showed in
\cite[Lemma 7.2]{michael:51} that not only is the relation
$\leq_\sigma$ transitive, but also that $X$ is weakly orderable with
respect to this linear order. Furthermore, for a connected space $X$
and $f\in\sel[\mathcal{F}(X)]$, he improved this result by showing in
\cite[Lemma 7.3]{michael:51} that $f(S)$ is the first $\leq_f$-element
of $S$ for every ${S\in \mathcal{F}(X)}$, where $\leq_f$ is the
order-like relation generated by the weak selection
$\sigma=f\uhr \mathcal{F}(X)$. In \cite{gutev-nogura:01a}, such a
selection $f\in \sel[\mathcal{F}(X)]$ was called
\emph{monotone}. There is nothing to prevent us from considering
monotone selections for general topological spaces. In fact, such
selections are naturally related to the so-called topologically
``well-orderable'' spaces, which were introduced in
\cite{engelking-heath-michael:68}. This relationship is discussed in
detail in the next section.\medskip

For a space $X$ and a monotone selection $f\in \sel[\mathcal{F}(X)]$,
the relation $\leq_f$ is a linear order on $X$ such that $p=f(X)$ is
the first $\leq_f$-element of $X$.  Hence, for a connected space $X$,
the set $X\setminus \{p\}=\left\{q\in X: p<_f q\right\}$ is an
interval in $X$ and is therefore also connected because $X$ is weakly
orderable with respect to $\leq_f$, see e.g.\ \cite[Theorem
1.3]{MR0339099}. Accordingly, in the realm of connected spaces, each
$p\in \ocs(X)$ is a \emph{noncut} point of $X$ and the cardinality of
$\ocs(X)$ doesn't exceed that of the noncut points of $X$. Thus, using
Eilenberg's result in \cite{eilenberg:41}, Michael obtained in
\cite[Proposition 7.8]{michael:51} that $\left|\ocs(X)\right|\leq 2$
for every connected spaces $X$. In the next section, we will refine
this result by showing that in this case, each point $p\in \ocs(X)$ is
in fact an \emph{endpoint} of the connected space $X$. Some supporting
examples will be also given in the same section. For instance, we will
show that there exists a connected space $X$ such that
$\ocs(X)\neq \emptyset$ and $p\notin \ocs(X)$ for some endpoint
$p\in X$.\medskip

The above property of $p=f(X)\in \ocs(X)$ for a monotone selection
${f\in\sel[\mathcal{F}(X)]}$ has another interesting
interpretation. Namely, a point $p\in X$ is called \emph{selection
  maximal} \cite{gutev-nogura:03a}, see also
\cite{garcia-ferreira-gutev-nogura-sanchis-tomita:99,gutev-nogura:00b},
if there exists a continuous selection $f$ for $\mathcal{F}(X)$ such
that ${f (S) = p}$ for every $S\in \mathcal{F}(X)$ with $p\in S$. In
this case, the selection $f$ is called \emph{$p$-maximal}.  Each
monotone selection $f\in \sel[\mathcal{F}(X)]$ is $f(X)$-maximal, but
there are simple examples of ``point''-maximal selections which are
not monotone. For instance, consider a space $Y$ such that
$\mathcal{F}(Y)$ has a continuous selection that is not monotone, and
add an isolated point to this space. The selection maximal elements of
the set $\ocs(X)$ were described in detail in \cite[Theorem
1.5]{Gutev2024a}. This result will play an interesting role in this
paper as well, see the proof of
\Cref{theorem-PM-Points-v21:1}.\medskip

For a general space $X$ with $\sel[\mathcal{F}(X)]\neq \emptyset$, the
set $\ocs(X)$ is finite precisely when $X$ has finitely many connected
components \cite[Theorem 1]{nogura-shakhmatov:97b}, see also
\cite[Corollary
2.7]{garcia-ferreira-gutev-nogura-sanchis-tomita:99}. The case when
$\ocs(X)$ is infinite doesn't suggest any definite result. This is
briefly discussed in \Cref{sec:select-totally-disc}, where we will
show that the closure of the set $\ocs(X)$ is always a totally
disconnected subset of $X$ (\Cref{theorem-Sequences-v3:1}). In the
same section, we will also show that $\ocs(X)\subseteq \ocs(Y)$ for
every closed subset $Y\subseteq X$ with $\ocs(X)\subseteq Y$
(\Cref{theorem-PM-Points-v21:1}). In \Cref{theorem-Sequences-v4:2} of
the last \Cref{sec:select-first-count} of this paper, we will show
that $\ocs(X)$ is a closed subset of every totally disconnected first
countable space $X$. Several supporting examples will be provided in
the same section (see
\Cref{example-PM-Points-v21:1,example-PM-Points-v21:2}). Finally, let
us also remark that several open questions are listed in
\Cref{sec:monotone-selections,sec:select-first-count}.

\section{Monotone Selections}
\label{sec:monotone-selections}

The starting point for the considerations in this section is the
following result obtained by Michael in \cite[Lemmas 7.2, 7.3 and
7.5.1]{michael:51}.

\begin{theorem}[\cite{michael:51}]
  \label{theorem-PM-Points-v17:1}
  For a space $X$ and a selection $f:\mathcal{F}(X)\to X$, the
  following properties hold.
  \begin{enumerate}[label=\upshape{(\roman*)}]
  \item\label{item:PM-Points-v17:1} If\/ $X$ is weakly orderable and\/
    $\leq$ is a compatible linear order on $X$ such that $f(S)$ is the
    first\/ $\leq$-element of\/ $S$ for every $S\in \mathcal{F}(X)$,
    then $f\in \sel[\mathcal{F}(X)]$.
  \item\label{item:PM-Points-v17:2} If\/ $X$ is connected and
    $f\in \sel[\mathcal{F}(X)]$, then\/ $\leq_f$ is a linear order on
    $X$ such that $X$ is weakly orderable with respect to this order
    and $f(S)$ is the first\/ $\leq_f$-element of $S$ for every
    $S\in \mathcal{F}(X)$.
  \end{enumerate}
\end{theorem}

Subsequently, Engelking, Heath and Michael
\cite{engelking-heath-michael:68} introduced the term
``\emph{topologically well-orderable}'' space to denote an orderable
space $X$ (by linear ordering $\leq$) such that every nonempty closed
subset of $X$ has a first $\leq$-element. In the same paper
\cite{engelking-heath-michael:68}, they extended this property to
suborderable spaces, saying that such a space $X$ is
\emph{topologically well-suborderable} (a \emph{topologically
  well-ordered subset}, in their terminology) if $X$ is a suborderable
space with respect to a linear order $\leq$ such that every nonempty
closed subset of $X$ has a first $\leq$-element.  In
\cite{gutev-nogura:01a}, the property was also extended to weakly
orderable spaces. Namely, we will say that a space $X$ is
\emph{topologically weakly well-orderable} (\emph{Sorgenfrey
  well-ordered} in the terminology of \cite{gutev-nogura:01a}) if $X$
is a weakly orderable space with respect to a linear order $\leq$ such
that every nonempty closed subset of $X$ has a first
$\leq$-element. Clearly, in these terms,
\Cref{theorem-PM-Points-v17:1} states that each topologically weakly
well-orderable space $X$ has a continuous selection for
$\mathcal{F}(X)$. Similarly, this theorem also implies that a
connected space $X$ is topologically weakly well-orderable if and only
if $\sel[\mathcal{F}(X)]\neq \emptyset$.\medskip

For a space $X$, the set $\mathcal{F}(X)$ is partially ordered with
respect to the usual set-theoretic inclusion. The role of the
selection in \Cref{theorem-PM-Points-v17:1} is naturally related to
this fact. Namely, let us recall that a selection
$f:\mathcal{F}(X)\to X$ is \emph{monotone} \cite{gutev-nogura:01a} if
$f(S)=f(T)$ for every $S,T\in \mathcal{F}(X)$ with
$f(T)\in S\subseteq T$. It was shown in \cite[Proposition
4.8]{gutev-nogura:01a} that the relation $\leq_f$ is a linear order on
$X$ for each monotone selection $f:\mathcal{F}(X)\to X$. This property
is crucial for the description of continuous monotone selections. To
this end, for a linear order $\leq$ on a set $X$, we will use the
standard notation for the intervals generated by this order. Namely,
$(\leftarrow, p)_{\leq}$ will stand for the $\leq$-interval of all
$x\in X$ with $x< p$; $(\leftarrow, p]_{\leq}$ for that of all
$x\in X$ with $x\leq p$; the $\leq$-intervals $(p,\to)_{\leq}$,
$[p,\to)_{\leq}$, $(p,q)_{\leq}$, $[p,q]_{\leq}$, etc., are likewise
defined.\medskip

The following characteristic property regarding continuity of a
monotone selection ${f:\mathcal{F}(X)\to X}$ was obtained in the proof
of \cite[Proposition 4.7]{gutev-nogura:01a}, but the proposition
itself is formulated incompletely since the requirement for weak
orderability of the space $X$ with respect to the linear order
$\leq_f$ was not explicitly stated.

\begin{proposition}
  \label{proposition-PM-Points-v15:1}
  Let $X$ be a space and $f:\mathcal{F}(X)\to X$ be a monotone
  selection. Then $f\in\sel[\mathcal{F}(X)]$ if and only if\/ $X$ is
  weakly orderable with respect to\/ $\leq_f$.
\end{proposition}

\begin{proof}
  The weak orderability of $X$ with respect to the linear order
  $\leq_f$ is equivalent to the property that the $\leq_f$-intervals
  $(p,\to)_{\leq_f}$ and $(\leftarrow,p)_{\leq_f}$, $p\in X$, are open
  in $X$. As shown in \cite[7.2.1 of Lemma 7.2]{michael:51}, these
  intervals are open for any continuous selection for
  $\mathcal{F}(X)$. This lemma from the aforementioned Michael's paper
  \cite{michael:51} was formulated for connected spaces and continuous
  weak selections, but connectedness did not play any role in the
  proof of this property. Conversely, for a monotone selection
  $f:\mathcal{F}(X)\to X$ such that the $\leq_f$-intervals
  $(p,\to)_{\leq_f}$ and $(\leftarrow,p)_{\leq_f}$, $p\in X$, are open
  in $X$, the continuity of $f$ was shown in the proof of
  \cite[Theorem 4.7]{gutev-nogura:01a}.
\end{proof}

Regarding \Cref{proposition-PM-Points-v15:1}, let us remark
that there are simple examples of discontinuous monotone
selections. For instance, take $X=[0,2)$ and define a monotone
selection $f:\mathcal{F}(X)\to X$ by $f(S)=\min S\cap [1,2)$ if
$S\cap [1,2)\neq \emptyset$, and $f(S)=\min S$ otherwise. Then
$\leq_f$ is a linear order on $X$ which is the usual order on $[0,1)$
and $[1,2)$, and $x\leq_f y$ for every $x\in [1,2)$ and $y\in
[0,1)$. This selection is not continuous.\medskip

Continuous monotone selections are implicitly present in
\Cref{theorem-PM-Points-v17:1}, as shown in \cite[Lemma 4.9 and
Corollary 4.10]{gutev-nogura:01a}.

\begin{theorem}[\cite{gutev-nogura:01a}]
  \label{theorem-PM-Points-v15:1}
  Let $X$ be a space. Then a continuous selection
  $f:\mathcal{F}(X)\to X$ is monotone if and only if\/ $\leq_f$ is a
  linear order on $X$ such that $f(S)$ is the first\/ $\leq_f$-element
  of\/ $S$ for every $S\in \mathcal{F}(X)$.
\end{theorem}

Thus, by \Cref{theorem-PM-Points-v17:1,theorem-PM-Points-v15:1}, each
continuous selection ${f:\mathcal{F}(X)\to X}$ for a connected space
$X$ is monotone.  Another aspect of the selection problem for
connected spaces is naturally related to an interesting
characterisation of topologically weakly well-orderable spaces in
\cite{gutev-nogura:01a}. To state it, for a linear order $\leq$ on a
set $X$, we will use $\mathcal{T}_\leq$ to denote the open interval
topology on $X$ generated by this order. The following result,
complementary to \Cref{theorem-PM-Points-v17:1}, was obtained in
\cite[Theorem 5.1]{gutev-nogura:01a}.

\begin{theorem}[\cite{gutev-nogura:01a}]
  \label{theorem-hsp-ver-8:3}
  Let $X$ be a weakly orderable space and $\leq$ be a compatible
  linear order on it. Then each nonempty closed subset of\/ $X$ has a
  first $\leq$-element if and only if
\begin{enumerate}[label=\upshape{(\roman*)}]
\item\label{item:PM-Points-v18:1} $(X, \mathcal{T}_\leq)$ is a
  topologically well-orderable space, and
\item\label{item:PM-Points-v18:2} If $p\in X$ is not the last
  $\leq$-element of $X$ and $U\subseteq X$ is an open set with
  $p\in U$, then there exists a point $q\in X$ such that $p< q$ and
  $[p,q)_\leq\subseteq U$.
\end{enumerate}
\end{theorem}

In \cite{gutev-nogura:01a}, any topology on a space $X$ satisfying
\ref{item:PM-Points-v18:1} and \ref{item:PM-Points-v18:2} of
\Cref{theorem-hsp-ver-8:3} is called a \emph{Sorgenfrey modification}
of $\mathcal{T}_\leq$, which may also explain the synonym of a
topologically weakly well-ordered space as ``Sorgenfrey
well-ordered''. As for the justification of this terminology, it goes
back to the Sorgenfrey line which can be viewed as a Sorgenfrey
modification of the topology on the interval $[0,1)$, namely the
following supporting example was given in \cite[Example
5.2]{gutev-nogura:01a}.

\begin{example}[\cite{gutev-nogura:01a}]
  \label{example-hsp-ver-8:2}
  The Sorgenfrey line is a Sorgenfrey modification of the interval
  $[0,1)$ and is therefore a topologically well-suborderable space.
\end{example}

Evidently, each \emph{ordinal space} (i.e.\ an ordinal equipped with
the open interval topology) is topologically well-orderable; also each
compact weakly orderable space is both orderable and topologically
well-orderable. Conversely, as shown in \cite[Lemmas 4.1, 4.2 and
4.3]{engelking-heath-michael:68}, each topologically well-orderable
space is locally compact. Moreover, each nonempty closed subset of a
topologically well-orderable space is also topologically
well-orderable, and every closed discrete subset of it is
countable. According to
\Cref{theorem-PM-Points-v15:1,theorem-hsp-ver-8:3}, this implies the
following consequence.

\begin{corollary}
  \label{corollary-PM-Points-v18:1}
  If\/ $X$ is a space, $f:\mathcal{F}(X)\to X$ is a continuous
  monotone selection and $\mathcal{T}_f=\mathcal{T}_{\leq_f}$ is the
  open interval topology generated by the linear order $\leq_f$, then
  $\left(X,\mathcal{T}_{f}\right)$ is a locally compact space such
  that each closed discrete subset of it is countable.
\end{corollary}

Finally, let us also comment on the importance of condition
\ref{item:PM-Points-v18:2} in \Cref{theorem-hsp-ver-8:3}. Namely, as
shown in \cite[Example 5.3]{gutev-nogura:01a}, there exists a
connected separable metrizable space $X$ which is weakly orderable,
$(X,\mathcal{T}_\leq)$ is a topologically well-orderable space for any
compatible linear order $\leq$ on it, but $X$ is not topologically
weakly well-orderable.\medskip

For a connected space $X$, \Cref{theorem-hsp-ver-8:3} also has some
interesting consequences regarding the set $\ocs(X)$, defined as in
\cref{eq:Sequences-v2:1}.  To this end, let us recall that a point
$p\in X$ in a connected space $X$ is \emph{cut} if $X\setminus\{p\}$
is not connected, otherwise $p\in X$ is called \emph{noncut}. In what
follows, we will use $\cut(X)$ for the set of the cut points of $X$,
and $\noncut(X)$\,---\,for that of the noncut points of $X$. An
alternative definition of a cut point $p\in X$ is that
$\{p\}=\overline{U}\cap \overline{V}$ for some (open) subsets
${U,V\subseteq X}$ with $X\setminus\{p\}=U\cup V$. It will be
convenient to call any such pair $(U,V)$ of subsets ${U,V\subseteq X}$
a \emph{$p$-cut} of $X$. A (cut) point $p\in X$ is said to
\emph{separate} $x,y\in X$ if $x\in U$ and $y\in V$ for some $p$-cut
$(U,V)$ of $X$. If $p$ separates $x$ and $y$, then neither $x$ nor $y$
separates the other two points, see \cite[Lemma 2.1]{MR0339099}.  In
these terms, a connected space $X$ is weakly orderable if and only if
among every three points of $X$ there is one that separates the other
two, see \cite[Theorem 4.1]{MR0339099} (in a footnote of
\cite{MR0235524}, the result was credited to
D. Zaremba-Szczepkowicz). Evidently, this property incorporates the
fact that $|\noncut(X)|\leq 2$ for every weakly orderable connected
space $X$. Furthermore, as remarked in the Introduction, any connected
space $X$ with $\sel[\mathcal{F}(X)]\neq \emptyset$ is weakly
orderable and $\ocs(X)\subseteq \noncut(X)$.\medskip

Here, using \Cref{theorem-PM-Points-v17:1} and
\ref{item:PM-Points-v18:2} from \Cref{theorem-hsp-ver-8:3}, we will
refine this result by examining the subtle difference between
\emph{cut points} and \emph{endpoints}. Namely, the \emph{order}
$\ord(p,X)$ of a point $p$ in a space $X$ in the sense of
Menger-Urysohn is the least cardinal $\lambda$ with the property that
every open set $U\subseteq X$ with $p\in U$ contains an open set
$V\subseteq X$ such that $p\in V$ and $|\partial V|\leq
\lambda$. Here, $\partial V$ is the \emph{boundary} of $V$. In these
terms, a space $X$ is \emph{zero-dimensional} if $\ord(p,X)=0$ for
every $p\in X$ or in other words, when it has a clopen base.\medskip

Let $X$ be a connected space and $p\in X$. Then $\ord(p,X)\geq 1$ and
$\ord(p,X)=1$ precisely when for every open set $U\subseteq X$ with
$p\in U$, there exists an open set $V\subseteq X$ such that
$p\in V\subseteq U$ and $\overline{V}= V\cup\{q\}$ for some point
$q\in X\setminus V$.  In the special case of $\ord(p,X)=1$, the point
$p\in X$ is called an \emph{endpoint} and we will use $\pend(X)$ to
denote the endpoints of $X$. It is easy to see that
$\pend(X)\subseteq \noncut(X)$ because ${\ord(p,X)\geq 2}$ for every
$p\in\cut(X)$. However, the converse is not true even when $X$ is
weakly orderable. For instance, $(0,0)$ is a noncut point of the
\emph{topological sine curve}
\begin{equation}
  \label{eq:PM-Points-v14:1}
  T=\{(0,0)\}\cup \left\{\left(t,\sin \tfrac1t\right): 0<t\leq
    1\right\},
\end{equation}
but it is not an endpoint of $T$. We now have the following natural
result.

\begin{proposition}
  \label{proposition-Sequences-v6:1}
  If\/ $X$ is a connected space, then\/ $\ocs(X)\subseteq \pend(X)$.
\end{proposition}

\begin{proof}
  Let $f\in \sel[\mathcal{F}(X)]$ and $U\subseteq X$ be an open set
  such that $A=X\setminus U\neq \emptyset$ and $p=f(X)\in U$. Next,
  define a map $h:X\to X$ by $h(x)=f(A\cup\{x\})$ for every $x\in
  X$. Then $h$ is continuous because both the set-valued mapping
  $X\ni x\longrightarrow A\cup\{x\}\in \mathcal{F}(X)$ and the map
  $f:\mathcal{F}(X)\to X$ are continuous with respect to the Vietoris
  topology on $\mathcal{F}(X)$. Finally, let $V=h^{-1}(U)$ and
  $H=\{x\in X: h(x)=x\}$.  Since $f$ is a selection for
  $\mathcal{F}(X)$ and $f(A\cup\{x\})=f(A)$ for every $x\in A$, it
  follows that $H\cap A=\{f(A)\}$ and $V\cup\{f(A)\}=H$. Furthermore,
  from the definition of the relation $\leq_f$ and
  \ref{item:PM-Points-v17:2} of \Cref{theorem-PM-Points-v17:1}, we
  also have that $p\in V=H\setminus A\subseteq U$.  Therefore,
  $\ord(p,X)=1$ because $H$ is closed in $X$ and $V$ is open in $X$.
\end{proof}

Regarding the proper place of \Cref{proposition-Sequences-v6:1}, let
us explicitly remark that not every endpoint of $X$ belongs to
$\ocs(X)$.

\begin{example}
  \label{example-Sequences-v8:1}
  Let
  $X=\left\{(s,0): s\in [-1,0]\right\}\cup \left\{(t,\sin 1/t):
    0<t\leq 1\right\}$. Then $X$ is a connected weakly orderable space
  with respect to the usual order on $[-1,1]$ because the projection
  $\pi:X\to [-1,1]$ is continuous and bijective. Furthermore, the
  maximal element of any $S\in \mathcal{F}(X)$ with respect to this
  order determines a continuous selection for $\mathcal{F}(X)$, which
  follows by using \Cref{theorem-PM-Points-v17:1} with the inverse
  linear order on $X$. In contrast, not every $S\in \mathcal{F}(X)$
  has a minimal element with respect to this order on $X$. For
  instance, such a set is
  ${S=\left\{\left(\frac2{(4k+1)},\sin\frac{(4k+1)\pi}2\right): k\in
      \N\right\}}$. Therefore, applying \Cref{theorem-PM-Points-v17:1}
  once again, there is no continuous selection $f:\mathcal{F}(X)\to X$
  with $f(X)= (-1,0)$.
\end{example}

\Cref{proposition-Sequences-v6:1} and \Cref{example-Sequences-v8:1}
suggest a complete description of the property of a connected weakly
orderable space $X$ to have a continuous selection for
$\mathcal{F}(X)$. It will be based on the following characteristic
property of those noncut points of $X$ that are endpoints, compare
with \Cref{theorem-hsp-ver-8:3}.

\begin{proposition}
  \label{proposition-Sequences-v8:1}
  Let $X$ be a weakly orderable connected space and\/ $\leq$ be a
  compatible linear order on it. If $p\in X$ is the first
  $\leq$-element of $X$, then $\ord(p,X)=1$ if and only if for every
  neighbourhood $U$ of $p$ there exists a point $q\in X\setminus\{p\}$
  such that $[p,q)_{\leq}\subseteq U$.
\end{proposition}

\begin{proof}
  In one direction, this is obvious because the interval
  $[p,q)_{\leq}=(\leftarrow,q)_{\leq}$ is open and
  $\overline{[p,q)_{\leq}}= [p,q]_{\leq}=[p,q)_{\leq}\cup\{q\}$ for
  every for $q\neq p$. Conversely, suppose that $\ord(p,X)=1$ and
  $U\subseteq X$ is an open set with $p\in U$. Then by definition,
  there exists an open set $V\subseteq X$ and a point
  $q\in X\setminus V$ such that $p\in V\subseteq U$ and
  $\overline{V}=V\cup \{q\}$. It now suffices to show that
  $[p,q)_{\leq}\subseteq V$. To this end, let us observe that
  $[p,q)_{\leq}$ is a connected set being an interval, see
  \cite[Theorem 1.3]{MR0339099}. Moreover,
  $V\cap [p,q)_{\leq}=\overline{V}\cap [p,q)_{\leq}$ is a nonempty
  clopen subset of $[p,q)_{\leq}$. Therefore,
  $V\cap [p,q)_{\leq}=[p,q)_{\leq}$ and accordingly
  $[p,q)_{\leq}\subseteq V$ as required.
\end{proof}

If $X$ is a weakly orderable connected space and $\leq$ is a
compatible linear order on it, then each interval $[q,\to)_{\leq}$,
$q\in X$, is also a connected subset of $X$ and $q$ is a noncut point
of this interval. In case each $q\in X$ is also an endpoint of the
interval $[q,\to){_\leq}$, we have the following result.

\begin{theorem}
  \label{theorem-Sequences-v6:1}
  Let $X$ be a weakly orderable connected space and $\leq$ be a
  compatible linear order on it. If\/ $p\in X$ is the first\/
  $\leq$-element of $X$, then $p\in \ocs(X)$ if and only if each\/
  $q\in X$ is an endpoint of the interval\/ $[q,\to)_{\leq}$.
\end{theorem}

\begin{proof}
  Let $f\in \sel[\mathcal{F}(X)]$ be such that $f(X)=p$. Then by
  \Cref{theorem-PM-Points-v17:1}, the relation $\leq_f$ is a
  compatible linear order on $X$ such that $f(S)$ is the first
  $\leq_f$-element of $S$ for every ${S\in \mathcal{F}(X)}$. Moreover,
  according to Eilenberg's result in \cite[Theorem II]{eilenberg:41},
  the linear orders $\leq$ and $\leq_f$ are identical. Therefore, by
  \Cref{proposition-Sequences-v6:1}, each $q\in X$ is an endpoint of
  the interval $[q,\to)_{\leq}$ because
  $f\left([q,\to)_{\leq}\right)=q$. Conversely, assume that $q\in X$
  is an endpoint of the interval $[q,\to)_{\leq}$ for every $q\in
  X$. The proof now consists of showing that every
  ${S\in \mathcal{F}(X)}$ has a first \mbox{$\leq$-element} $f(S)$
  because by \Cref{theorem-PM-Points-v17:1}, the map
  $f:\mathcal{F}(X)\to X$, defined in this way, is a continuous
  selection for $\mathcal{F}(X)$ with $p=f(X)\in \ocs(X)$. So, take an
  element $S\in \mathcal{F}(X)$ with $p\notin S$. Since $p$ is an
  endpoint of $X$ and $p\in X\setminus S$, it follows from
  \Cref{proposition-Sequences-v8:1} that $S\subseteq [q,\to)_{\leq}$
  for some $q\in X\setminus\{p\}$. Therefore, we also have
  $V=\bigcup_{x\in S}(x,\to)_{\leq}\subseteq [q,\to)_{\leq}$ and hence
  $\overline{V}\setminus V\neq \emptyset$ because $V$ is open but not
  closed in the connected space $X$. In fact,
  $\overline{V}\setminus V$ is a singleton because
  $(\leftarrow, y)_{\leq}\cap V=\emptyset$ for every
  $y\in \overline{V}\setminus V$. Thus, there exists a point
  $f(S)\in X$ such that
  $\overline{V}=\left[f(S),\to\right)_{\leq}$. However, by assumption,
  $f(S)$ is an endpoint of $\left[f(S),\to\right)_{\leq}$ and,
  therefore, $U\cap S\neq \emptyset$ for every open set $U\subseteq X$
  that contains the point $f(S)$. Accordingly, $f(S)\in S$ and clearly
  it is the first $\leq$-element of $S$. The proof is complete.
\end{proof}

All statements in \Cref{theorem-Sequences-v6:1} appear to be the best
possible. The assumption that each $q\in X$ is an endpoint of the
interval $[q,\to)_{\leq}$ cannot be dropped or even weakened to some
points of $X$. One counterexample is the topological sine curve $T$
defined as in \cref{eq:PM-Points-v14:1} and equipped with the
compatible linear order $\leq$ of the underlying interval
$[0,1]$. Namely, $\sel[\mathcal{F}(T)]\neq \emptyset$ and each
$q\in T\setminus\{(0,0)\}$ is an endpoint of the interval
$[q,\to)_{\leq}$, but $(0,0)\notin \ocs(T)$. Another counterexample is
the space $X$ defined as in \Cref{example-Sequences-v8:1} and equipped
with the compatible linear order $\leq$ on the underlying interval
$[-1,1]$. We have again that $\sel[\mathcal{F}(X)]\neq \emptyset$, but
now $(-1,0)\notin \ocs(X)$ and yet $q\in X$ is an endpoint of the
interval $[q,\to)_{\leq}$ for every $q\in
X\setminus\{(0,0)\}$. Finally, regarding the requirement that $p\in X$
is the first $\leq$-element of $X$, let us point out the following
natural consequence of \Cref{theorem-Sequences-v6:1}.

\begin{corollary}
  \label{corollary-PM-Points-v15:1}
  Let $X$ be a weakly orderable connected space and $\leq$ be a
  compatible linear order on it. If
  $p\in \ocs\left([p,\to)_{\leq}\right)$ for some $p\in X$, then
  $q\in \ocs\left([q,\to)_{\leq}\right)$ for every $q\in X$ with
  $p\leq q$.
\end{corollary}

It should be remarked that, unlike \Cref{theorem-Sequences-v6:1}, the
property in \Cref{corollary-PM-Points-v15:1} is valid for any weakly
orderable connected space $X$ without explicitly requiring that
${\sel[\mathcal{F}(X)]\neq \emptyset}$. For instance, it is valid for
the real line $\R$ and the usual linear order $\leq$ on it, but
$\sel[\mathcal{F}(\R)]=\emptyset$, as shown in \cite[Proposition
5.1]{engelking-heath-michael:68}. \medskip

We conclude this section with some open questions. The following
question was raised in \cite[Question 388]{gutev-nogura:06b} and
subsequently in \cite[Problem 2.18]{gutev-2013springer}.

\begin{question}[\cite{gutev-2013springer,gutev-nogura:06b}]
  \label{question-PM-Points-v19:1}
  Let $X$ be a space that has a continuous selection for
  $\mathcal{F}(X)$. Then is $X$ topologically weakly well-orderable?
\end{question}

In \cite[Theorem 2.7]{hrusak-martinez:09}, Michael Hru\v{s}\'ak and
Iv\'an Mart{\'{\i}}nez-Ruiz constructed a separable, first countable
and locally compact space which admits a continuous weak selection but
is not weakly orderable. In contrast, the question of whether there
exists a space $X$ that has a continuous selection for
$\mathcal{F}(X)$ but is not weakly orderable is still open. Since
every topologically weakly well-orderable space is also weakly
orderable, \Cref{question-PM-Points-v19:1} is open even in the
following special case.

\begin{question}
  \label{question-PM-Points-v19:2}
  Let $X$ be a weakly orderable space that has a continuous selection
  for $\mathcal{F}(X)$. Then is $X$ topologically weakly
  well-orderable?
\end{question}

In \cite[Theorem 1.2]{douwen:90}, Eric van Douwen showed that every
countably compact space $X$ with
$\sel[\mathcal{F}_2(X)]\neq \emptyset$ must be sequentially compact,
hence for these spaces countable compactness is equivalent to
sequential compactness. By resolving a question regarding the role of
countable compactness in the theory of continuous weak selection
raised in \cite[Question 1]{MR3122363}, it was obtained in
\cite[Theorems 1.5]{Motooka2019} that every countably compact space
with a continuous weak selection is weakly orderable. This gives a
further justification to the following question posed in
\cite[Question 389]{gutev-nogura:06b} and \cite[Problem
3.15]{gutev-2013springer}.

\begin{question}
  \label{question-PM-Points-v19:3}
  Let $X$ be a sequentially compact space that has a continuous
  selection for $\mathcal{F}(X)$. Then is $X$ topologically weakly
  well-orderable?
\end{question}

Regarding the selection problem for compact-like spaces, several
authors (see \cite{artico-marconi-pelant-rotter-tkachenko:02,
  douwen:90, garcia-ferreira-sanchis:04, glicksber:59, miyazaki:01b,
  venkataraman-rajagopalan-soundararajan:72}) contributed to the
fundamental result that every pseudocompact space with a continuous
weak selection is suborderable.  In particular, for a Tychonoff space
with a continuous weak selection, countable compactness,
pseudocompactness and sequential compactness are equivalent
topological properties. The interested reader is referred to
\cite{gutev-2013springer} where these results were discussed in detail
and where in \cite[Problem 3.13]{gutev-2013springer} the following
natural question was posed.

\begin{question}
  \label{question-PM-Points-v19:4}
  Let $X$ be a pseudocompact space that has a continuous selection for
  $\mathcal{F}(X)$. Then is it true that $X$ is a topologically
  well-suborderable space?
\end{question}

To clarify another aspect of \Cref{question-PM-Points-v19:4}, let us
point out that according to \cite[Proposition
3.14]{gutev-2013springer}, a pseudocompact space $X$ is topologically
well-suborderable if and only if it is a topologically well-ordered
subset (in the sense of \cite{engelking-heath-michael:68}) of the
\v{C}ech-Stone compactification $\beta X$ of $X$.

\section{Selections and Totally Disconnected Spaces}
\label{sec:select-totally-disc}

A space $Y$ is \emph{totally disconnected} if each singleton of $Y$ is
an intersection of clopen subsets of $Y$. It was shown in
\cite[Theorem 1.5]{gutev-nogura:00b} that a space $X$ is totally
disconnected whenever the set $\ocs(X)$ is dense in $X$. In this
section, we will give a very simple proof of the following natural
generalisation of this result. 

\begin{theorem}
  \label{theorem-Sequences-v3:1}
  If\/ $X$ is a space, then\/ $\overline{\ocs(X)}$ is a totally
  disconnected subset of\/ $X$.
\end{theorem}

The proof of \Cref{theorem-Sequences-v3:1} is based on several known
results and the interpretation that a space $Y$ is totally
disconnected precisely when each quasi-component of $Y$ is a
singleton. To this end, for a partition $\mathcal{P}$ of a set $X$ and
a point $p\in X$, let $\mathcal{P}[p]\in \mathcal{P}$ be the unique
element of $\mathcal{P}$ with ${p\in \mathcal{P}[p]}$.  The
\emph{components} (sometimes called \emph{connected components}) of a
space $X$ are the maximal connected subsets of $X$. They form a closed
partition $\mathcal{C}$ of $X$, and each element
$\mathcal{C}[p]\in \mathcal{C}$ corresponding to a point $p\in X$ is
called the \emph{component} of this point. The \emph{quasi-component}
$\mathcal{Q}[p]$ of a point $p\in X$ is the intersection of all clopen
subsets of $X$ containing this point. The quasi-components also form a
closed partition $\mathcal{Q}$ of $X$, thus they are simply called the
\emph{quasi-components} of $X$.  It is obvious that
$\mathcal{C}[p]\subseteq \mathcal{Q}[p]$, $p\in X$, but the converse
is not necessarily true. However, as shown in \cite[Theorem
4.1]{gutev-nogura:00b} and refined in \cite[Theorem
6.1]{gutev-nogura:09a}, the closed partitions $\mathcal{C}$ and
$\mathcal{Q}$ coincide for any space $X$ that has a continuous weak
selection.

\begin{theorem}[\cite{gutev-nogura:00b,gutev-nogura:09a}]
  \label{theorem-components-selection-relation}
  If\/ $X$ is a space with $\sel[\mathcal{F}_2(X)]\neq \emptyset$,
  then $\mathcal{C}[p]=\mathcal{Q}[p]$ for every point $p\in X$.
\end{theorem}

The following special sets, defined in \cite{gutev:07a}, are based on
a construction given by Purisch in \cite{purisch:77}. 

\begin{definition}
  [\cite{gutev:07a,purisch:77}]
  \label{definition-purisch-set}
  Let $X$ be a space such that
  $\left|\noncut(\mathcal{C}[p])\right|\leq 2$ for every $p\in X$. We
  shall say that a subset $Z\subseteq X$ is a \emph{Purisch set} if
  $\noncut(X)\subseteq Z$ and
  $\left|Z\cap \cut(\mathcal{C}[p])\right|=1$ for every point $p\in X$
  such that $\cut(\mathcal{C}[p])\neq \emptyset$ and
  $\left|\noncut(\mathcal{C}[p])\right|\leq 1$.
\end{definition}

Let us point out the following basic properties of Purisch subsets.

\begin{proposition}
  \label{proposition-Sequences-v3:1}
  Let $X$ be a space with $\sel[\mathcal{F}_2(X)]\neq \emptyset$. Then
  $\left|\noncut(\mathcal{C}[p])\right|\leq 2$ for every $p\in X$ and,
  in particular, $X$ has at least one Purisch subset. Moreover, any
  Purisch subset of $X$ is totally disconnected and closed.
\end{proposition}

\begin{proof}
  Since $\sel[\mathcal{F}_2(X)]\neq \emptyset$, it follows from
  \cite[Corollary 2.7]{gutev:07a} that
  $\left|\noncut(\mathcal{C}[p])\right|\leq 2$ and
  $\cut(\mathcal{C}[p])$ is open in $X$, for every $p\in X$. The fact
  that any Purisch subset of $X$ is closed was obtained in
  \cite[Proposition 3.3]{gutev:07a}, it is a consequence of
  \Cref{definition-purisch-set} and the property that the cut points
  of each component of $X$ form an open subset of $X$. The fact that
  any Purisch subset of $X$ is totally disconnected was obtained in
  \cite[Proposition 3.2]{gutev:07a}, it is a consequence of
  \Cref{theorem-components-selection-relation} and
  \Cref{definition-purisch-set}.
\end{proof}

Based on these properties, the proof of \Cref{theorem-Sequences-v3:1}
ends with the following simple observation.

\begin{proposition}
  \label{proposition-Sequences-v4:1}
  If\/ $X$ is a space with $\sel[\mathcal{F}(X)]\neq \emptyset$,
  then\/ $\ocs(X)\subseteq Z$ for any Purisch subset $Z\subseteq X$.
\end{proposition}

\begin{proof}
  Take a point $p\in \ocs(X)$ and a Purisch subset $Z\subseteq X$, see
  \Cref{proposition-Sequences-v3:1}. Then by \cite[Proposition
  5.1]{Gutev2024a}, $p\in \noncut(\mathcal{C}[p])$. Hence, by
  \Cref{definition-purisch-set}, $p\in Z$.
\end{proof}

There are very simple examples of connected spaces $X$ such that
$\sel[\mathcal{F}(X)]\neq \emptyset$ and yet $\ocs(X)$ is not equal to
any Purisch subset of $X$. In fact, according to
\Cref{proposition-Sequences-v6:1}, any connected space $X$ with
$\sel[\mathcal{F}(X)]\neq \emptyset$ and $\pend(X)\neq \noncut(X)$ is
an example in this regard. Similarly, there are simple examples in the
case of totally disconnected spaces. They are based on the following
characteristic property of total disconnectedness in the realm of
spaces that have continuous weak selections.

\begin{corollary}
  \label{corollary-Sequences-v4:1}
  Let $X$ be a space with $\sel[\mathcal{F}_2(X)]\neq\emptyset$. Then
  $X$ is totally disconnected if and only if\/ $X$ is a Purisch subset
  of itself.
\end{corollary}

\begin{proof}
  By \Cref{proposition-Sequences-v3:1}, $X$ has a Purisch subset and
  $X$ is totally disconnected if it is a Purisch subset of
  itself. Conversely, if $X$ is totally disconnected, then by
  \Cref{theorem-components-selection-relation}, the connected
  components of $X$ are singletons. Hence, by
  \Cref{definition-purisch-set}, $X=Z$.
\end{proof}

To discuss the proper place of \Cref{theorem-Sequences-v3:1}, let us
first mention the following simple observation, which is complementary
to \cite[Theorem 1.5]{gutev-nogura:00b}.

\begin{proposition}
  \label{proposition-PM-Points-v21:1}
  The set\/ $\ocs(X)$ is dense in a space $X$ if and only if\/
  $\ocs(Y)\subseteq\ocs(X)$ for every closed subset $Y\subseteq X$
  with $\ocs(X)\subseteq Y$.
\end{proposition}

\begin{proof}
  If $\ocs(X)$ is dense in $X$ and $Y\subseteq X$ is a closed set with
  $\ocs(X)\subseteq Y$, then $Y=X$ and trivially
  $\ocs(Y)=\ocs(X)$. Conversely, let $\ocs(Y)\subseteq\ocs(X)$ for
  every closed subset $Y\subseteq X$ with $\ocs(X)\subseteq Y$. To
  show that $\ocs(X)$ is dense in $X$, let us assume that this is not
  true, and let us take a point $p\in X\setminus
  \overline{\ocs(X)}$. Since $p$ is an isolated point of the closed
  set $Y=\overline{\ocs(X)}\cup\{p\}$, this implies
  $p\in \ocs(Y)\subseteq \ocs(X)$. Obviously, this is impossible.
\end{proof}

Another way to state \Cref{proposition-PM-Points-v21:1} is that the
set $\ocs(X)$ is dense in a space $X$ if and only if $\ocs(Y)=\ocs(X)$
for every closed subset $Y\subseteq X$ with $\ocs(X)\subseteq Y$. Our
next result will show that the inclusion $\ocs(X)\subseteq \ocs(Y)$ is
valid in general for every closed set $Y\subseteq X$ with
$\ocs(X)\subseteq Y$. To this end, let us recall that a point $p\in X$
of a space $X$ is \emph{countably-approachable} \cite{gutev:05a} if it
is either isolated or has a countable clopen base in $\overline{U}$
for some open set $U\subseteq X\setminus\{p\}$ with
$\overline{U}=U\cup\{p\}$. The non-isolated countably-approachable
points were called \emph{$\omega$-approachable} in \cite{gutev:05a}.
One can easily see that a point $p\in X$ is $\omega$-approachable
precisely when $p$ is the limit of a pairwise disjoint sequence
$S_n\subseteq X\setminus\{p\}$, $n\in \N$, of nonempty clopen subsets
of $X$. Here, by ``limit'' we mean that the sequence
$\left\{S_n\right\}\subseteq \mathcal{F}(X)$ is $\tau_V$-convergent to
$\{p\}$.

\begin{theorem}
  \label{theorem-PM-Points-v21:1}
  If $X$ is a space, then\/ $\ocs(X)\subseteq \ocs(Y)$ for every
  closed subset $Y\subseteq X$ with $\ocs(X)\subseteq Y$.
\end{theorem}

\begin{proof}
  Take a closed set $Y\subseteq X$ with $\ocs(X)\subseteq Y$. In case
  $\ocs(X)=\emptyset$, there is nothing to prove. Otherwise, take a
  point $p\in \ocs(X)$. If this point is isolated in $X$, then it is
  also isolated in $Y$ and trivially $p\in \ocs(Y)$. Assume that
  $p\in X$ is $\omega$-approachable. Then, as stated above, $p$ is the
  limit of a pairwise disjoint sequence
  $S_n\subseteq X\setminus\{p\}$, ${n\in \N}$, of nonempty clopen
  subsets of $X$. However, by \cite[Lemma 2.1]{gutev-nogura:00b}, each
  nonempty clopen subset of $X$ contains a point of
  $\ocs(X)$. Therefore, $p$ is also $\omega$-approachable in $Y$
  because $S_n\cap Y\neq \emptyset$ for every $n\in \N$. Thus, by
  \cite[Lemma 4.2]{gutev:05a}, $p\in \ocs(Y)$. Finally, if $p\in X$ is
  not countably approachable, then \cite[Theorem 1.5]{Gutev2024a}
  implies that $\mathcal{F}(X)$ has a $p$-maximal selection
  $f\in\sel[\mathcal{F}(X)]$. Accordingly, $f(Y)=p$ and the proof is
  complete.
\end{proof}
 
We conclude this section with the following two remarks.

\begin{remark}
  \label{remark-PM-Points-v24:1}
  Let $X$ be a space such that
  $\left|\noncut(\mathcal{C}[p])\right|\leq 2$ for every $p\in X$. The
  definition of a Purisch subset $Z\subseteq X$ in \cite{gutev:07a} is
  a slight modification of similar sets defined by Purisch in
  \cite{purisch:77}. The difference is for a point $p\in X$ such that
  $\mathcal{C}[p]\subseteq X$ is non-degenerate and open in $X$. In
  this case, in terms of \cite{purisch:77},
  $\left|Z\cap \mathcal{C}[p]\right|=1$ without explicitly requiring
  that $\noncut(\mathcal{C}[p])\subseteq Z$. In other words, in the
  terminology of \cite{purisch:77},
  $\left|Z\cap \mathcal{C}[p]\right|=1$ precisely when
  $\mathcal{C}[p]$ is a singleton or a clopen subset of $X$. In
  contrast, the modification in \cite{gutev:07a} aims to keep a track
  on the noncut points of $\mathcal{C}[p]$. Thus, in terms of
  \Cref{definition-purisch-set}, if $\mathcal{C}[p]$ is a
  non-degenerate clopen component of $X$ with
  $\noncut(\mathcal{C}[p])\neq \emptyset$, then
  $\noncut(\mathcal{C}[p])\subseteq Z$ and
  $\left|Z\cap \mathcal{C}[p]\right|=2$. Finally, let us also remark
  that in both interpretations, these special subsets are uniquely
  determined in the sense that any two such sets are
  homeomorphic. This was explicitly stated in
  \cite{purisch:77}. Similarly, if $Z_1,Z_2\subseteq X$ are Purisch
  subsets in the sense of \Cref{definition-purisch-set}, then
  $\noncut(X)\subseteq Z_1\cap Z_2$ and each subset
  $Z_i\setminus \noncut(\mathcal{C}[p])$, $i=1,2$, consists of
  isolated points of $Z_i$. Since
  $Z_1\setminus \noncut(\mathcal{C}[p])$ and
  $Z_2\setminus \noncut(\mathcal{C}[p])$ have the same cardinality,
  this shows that $Z_1$ and $Z_2$ are homeomorphic.
\end{remark}

\begin{remark}
  \label{remark-PM-Points-v21:1}
  As mentioned before, a point $p\in X$ in a connected space $X$ is
  cut precisely when $X\setminus\{p\}=U\cup V$ for some (open) subsets
  $U,V\subseteq X$ such that $\overline{U}\cap
  \overline{V}=\{p\}$. Extending this interpretation to an arbitrary
  space $X$, a point $p\in X$ is called \emph{cut}
  \cite{gutev-nogura:03a}, see also \cite{gutev:00e,gutev-nogura:00d},
  if $\{p\}=\overline{U}\cap \overline{V}$ for some (open) subsets
  ${U,V\subseteq X}$ with $X\setminus\{p\}=U\cup V$. Cut points were
  also introduced in \cite{Dow2008}, where they were called
  \emph{tie-points}. As shown in \cite[Theorem 1.2]{Gutev2024a}, a
  point $p\in \ocs(X)$ is a cut point of a space $X$ precisely when it
  is $\omega$-approachable in $X$.
\end{remark}

\section{Selections and First Countable Spaces}
\label{sec:select-first-count}

A family $\mathcal{P}$ of open sets of a space $X$ is a
\emph{$\pi$-base} (called also a \emph{pseudobase}, Oxtoby
\cite{oxtoby:60}) if every nonempty open subset of $X$ contains some
nonempty member of $\mathcal{P}$. The following result was obtained in
\cite[Theorem 2.1]{gutev:05a}.

\begin{theorem}[\cite{gutev:05a}]
  \label{theorem-background-orbits}
  Let $X$ be a space such that
  $\sel[\mathcal{F}(X)]\neq\emptyset$. Then the set\/ $\ocs(X)$ is
  dense in $X$ if and only if\/ $X$ has a clopen $\pi$-base.
\end{theorem}

In \cite[Corollary 4.5]{gutev:05a}, based on this result, it was shown
that if $X$ is a regular first countable space and $\ocs(X)$ is dense
in $X$, then $\ocs(X)=X$. Here, we will show that the condition on $X$
to be regular is not necessary, in fact we will obtain the following
more general result.

\begin{theorem}
  \label{theorem-Sequences-v4:2}
  If\/ $X$ is a totally disconnected first countable space, then\/
  $\ocs(X)$ is closed in $X$.
\end{theorem}

\begin{proof}
  Let us assume that a point $p\in X$ is the limit of a sequence
  $\left\{x_n\right\}\subseteq \ocs(X)\setminus \{p\}$. The proof now
  consists in showing that such a point $p\in X$ is
  $\omega$-approachable. To this end, we take a decreasing local base
  $\{U_n\}$ at this point. Then $U_1$ contains some term $x_{n_1}$ of
  the sequence $\{x_n\}$. If $x_{n_1}$ is an isolated point, take
  $S_1=\left\{x_{n_1}\right\}$. Otherwise, if $x_{n_1}$ is not
  isolated, it is the limit of a nontrivial sequence. Since $X$ is
  totally disconnected and $x_{n_1}\in \ocs(X)$, it follows from
  \cite[Theorem 1.2 and Corollary 1.3]{Gutev2024a} that $x_{n_1}$ is
  $\omega$-approachable. Hence, by definition, $U_1\setminus\{p\}$
  contains a nonempty clopen set $S_1$ because
  $x_{n_1}\in U_1\setminus\{p\}$. Next, take $n_2>n_1$ such that
  $x_{n_2}\in U_{n_2}$ and $U_{n_2}\cap S_1=\emptyset$. Then for the
  same reason as before, $U_{n_2}\setminus\{p\}$ contains a nonempty
  clopen subset $S_2\subseteq X$. Thus, by induction, there exists a
  subsequence $\left\{U_{n_k}\right\}$ of $\{U_n\}$ and a sequence
  $\left\{S_k\right\}$ of nonempty clopen subsets of $X$ such that
  $S_k\subseteq U_{n_k}\setminus U_{n_{k+1}}$ for every
  $k\in\N$. Accordingly, $\{S_k\}$ is convergent to $p$, and its
  elements do not contain the point $p$. Hence, by definition, $p$ is
  $\omega$-approachable. Therefore, by \cite[Lemma 4.2]{gutev:05a},
  $p\in \ocs(X)$.
\end{proof}

We conclude this paper with two examples concerning the role of total
disconnectedness and first countability in
\Cref{theorem-Sequences-v4:2}, and some additional remarks regarding
the role of clopen $\pi$-bases in \Cref{theorem-background-orbits}.

\begin{example}
  \label{example-PM-Points-v21:1}
  There exists a totally disconnected compact orderable space $X$ such
  that $\ocs(X)$ is dense in $X$ but not closed in $X$. Such an
  example can be constructed as follows. For limit ordinals $\lambda$
  and $\mu$, as in \cite{fujii-nogura:99}, we will use
  $L(\lambda, \mu)$ to denote the \emph{wedge sum}
  $(\lambda + 1) \vee_{\lambda=\mu} (\mu + 1)$. We can now take
  $X=L(\omega_1,\omega_1)$, where $\omega_1$ is the first uncountable
  ordinal. Also, for convenience, let $p=\omega_1\in X$ be the point
  at which the ordinals $\omega_1$ and $\omega_1$ are identified. Then
  $\ocs(X)=X\setminus\{p\}$. Indeed, if $q\in X$ and $q\neq p$, then
  $q$ is either an isolated point or a countable limit ordinal. Hence,
  in a trivial way, $q\in \ocs(X)$. The fact that $p\notin \ocs(X)$
  follows from \cite[Theorem 1.3]{Gutev2024a} because $p$ is a cut
  point of $X$ (see \Cref{remark-PM-Points-v21:1}), but is not a limit
  of a sequence of points of $X \setminus\{p\}$.
\end{example}

\begin{example}
  \label{example-PM-Points-v21:2}
  There exists a separable metrizable space $X$ with
  $\overline{\ocs(X)}\setminus \ocs(X)\neq\emptyset$. The example is
  based on the following modification of the topological sine
  curve. Namely, if $t\geq 1$, then $\sin t\geq 0$ if and only if
  $t\in [2k\pi,(2k+1)\pi]$ for some $k\in\N$. Using this and following
  \cref{eq:PM-Points-v14:1}, let
  $\Delta_k=\left[\frac1{(2k+1)\pi},\frac1{2k\pi}\right]$ for every
  $k\in \N$. Next, set $p=(0,0)$ and
  $X=\{p\}\cup \left\{\left(t,\sin \frac1t\right): t\in \bigcup_{k\in
      \N}\Delta_k\right\}$. Since each set
  $U_k=\left\{\left(t,\sin \frac 1t\right): t\in \Delta_k\right\}$,
  $k\in \N$, is compact, orderable and clopen in $X$, the endpoints
  ${x_k=\left(\frac1{(2k+1)\pi},0\right),
    y_k=\left(\frac1{2k\pi},0\right)\in U_k}$ belong to
  $\ocs(X)$. Moreover, it is evident that
  $\lim_{k\to \infty}x_k=p=\lim_{k\to \infty}y_k$. This implies that
  $p\in X$ is a cut point of $X$ in the sense of
  \Cref{remark-PM-Points-v21:1}. However, $p\in X$ is not an
  $\omega$-approachable point of $X$ because the set
  $F=\left\{\left(\frac2{(4k+1)\pi}, 1\right): k\in \N\right\}$ is
  closed in $X$ and $X\setminus F$ doesn't contain any clopen
  neighbourhood of $p$. Accordingly, by \cite[Theorem
  1.2]{Gutev2024a}, $p\notin \ocs(X)$.
\end{example}

Regarding the role of clopen $\pi$-bases in
\Cref{theorem-background-orbits}, the following question was posed in
\cite[Question 391]{gutev-nogura:06b}.

\begin{question}
  \label{question-Sequences-v2:1}
  For a totally disconnected space $X$ with
  $\sel[\mathcal{F}(X)]\neq \emptyset$, is the set $\ocs(X)$ dense in
  $X$? In other words, is it true that each totally disconnected space
  $X$ with $\sel[\mathcal{F}(X)]\neq \emptyset$ has a clopen
  $\pi$-base?
\end{question}

\Cref{question-Sequences-v2:1} is open even in the special case of
(separable) metrizable spaces. Regarding this special case, let us
remark that each totally disconnected second countable space is weakly
orderable, see \cite[Remarks 5.5 and 5.6]{gutev:07a}. This leads to
the following further refinement of \Cref{question-Sequences-v2:1}.

\begin{question}
  \label{question-PM-Points-v21:1}
  For a totally disconnected weakly orderable space $X$ with
  $\sel[\mathcal{F}(X)]\neq \emptyset$, is the set $\ocs(X)$ dense in
  $X$?
\end{question}

\end{document}